\documentclass[preprint]{amsart}
\usepackage[utf8]{inputenc}
\usepackage[english]{babel}
\usepackage{graphicx}
\usepackage{amsmath}
\usepackage{amsfonts}
\usepackage{amssymb}
\usepackage[numbers]{natbib}
\usepackage[none]{hyphenat}
\usepackage{amssymb}
\usepackage{enumerate}
\usepackage[mathscr]{euscript}
\newtheorem{theorem}{Theorem}
\newtheorem{lem}{Lemma}
\newtheorem{example}{Example}

\usepackage{ragged2e}

\begin{document}

\title{A stochastic comparison study for the smallest and largest ordered statistic from Weibull-G and Gompertz Makeham distribution}

\author{Madhurima Datta and Nitin Gupta\\
Indian Institute of Technology Kharagpur}

\address{Department of Mathematics, Indian Institute of Technology Kharagpur}

\email[M. Datta]{ madhurima92.datta@iitkgp.ac.in}
 \email[N. Gupta]{ nitin.gupta@maths.iitkgp.ac.in}
 
\maketitle
\begin{abstract}
In this paper, we have discussed the stochastic comparison of the smallest and largest ordered statistic from independent heterogeneous Weibull-G random variables and Gompertz Makeham random variables. We compare systems arising from taking different model parameters and obtain stochastic ordering results under the condition of multivariate chain majorization. Using the notion of vector majorization, we compare different systems and obtain stochastic ordering results.
\end{abstract}


MSC 62N05, 90B25

\section{Introduction}
Order statistics have been widely studied from the past few decades. It is the arrangement of random variables according to their strength. Let $X_1, \ldots, X_n$ be a set of random variables then the increasing arrangement of these random variables is $X_{1:n} \leq X_{2:n} \ldots \leq X_{n:n}$, here $X_{1:n}$ is the smallest ordered statistic and $X_{n:n}$ is the largest ordered statistic. The terms in the form of $X_{k:n}$ are the k-th ordered statistic. An interpretation of these ordered statistic can be observed in terms of system whose components are these n random variables. The smallest ordered statistic is the minimum of all the n random variables, i.e., $X_{1:n} = \min(X_1, X_2, \ldots, X_n)$, this represents the structure of a series system. Similarly the largest ordered statistic $X_{n:n} = \max(X_1, X_2, \ldots, X_n)$ represents the parallel system and $X_{k:n}$ represents the $n-k+1$ out of $n$ system. The parallel and series systems are the primary units of complex coherent systems. Moreover, they are often used in statistics, applied probability, auction theory, actuarial science etc. \\
In this paper we shall focus on the study of stochastic comparison of smallest and largest ordered statistic whose components follow Weibull- G distribution and Gompertz Makeham distribution. Weibull and Gompertz Makeham distributions are widely used to study the nature of human death distribution as they contain minimum number of parameters and the distribution functions are in closed form. Both of them are limit distributions of extreme value theory and a comparative analysis of both the distributions have been studied by \cite{15}. Firstly let us understand how the Weibull-G distribution evolved. The distribution originated from T-X (Transformed transformer) family of distributions, proposed by Alzaatreh et al.\cite{6} as a new method of generating continuous distributions. The T-X family of distributions can fit highly left tailed, right tailed, unimodal, bimodal distributions. This family of distributions was greatly inspired from the Beta- Normal distributions proposed by Eugene et al.\cite{5}. Let $X$ and $T$ be two continuous random variables such that the density function, distribution function(cdf) and survival function(sf) of $X$ are $f(x), F(x), \overline{F}(x)$ respectively, and $T$ has support on $[a, \infty)$ with density function $u(x)$. Let $X_{T}$ be a random variable following T-X family of distributions, the cdf of $X_{T}$ is
\begin{equation}\label{TX}
H_{TX}(x)= \int_a^{w(F(x))}u(t) dt, x \in \mathbb{R},
\end{equation}
where $w(F(x))$ is continuously differentiable and monotonically non-decreasing in the interval $[a, \infty)$. Let $X$ follow a scale family of distribution say $X \sim F(\gamma x)$, $\gamma >0$ is the scale parameter. The Weibull-G distribution is generated when the random variable $T$ follows the Weibull distribution ($T \sim W(\alpha, \beta)$, $\alpha >0, \beta >0$) with parameters $\alpha$ and $\beta$ and the function $w(F(x)) = \dfrac{F(x)}{1-F(x)}$ in $\eqref{TX}$. The resulting distribution function of $X_{WG}$ is
$$H_{WG}(x) = 1- e^{-\alpha\left(\dfrac{F(\gamma x)}{1- F(\gamma x)}\right)^{\beta}}, ~x >0, \alpha >0, \beta >0, \gamma>0,$$
here $\alpha, \gamma$ are the scale parameters and $\beta$ is the shape parameter. We shall denote the distribution as $W-G(\alpha, \beta, \gamma)$.\\
The generalised distributions are extremely useful for practical purposes as they have more number of parameters. The stochastic comparison of such distributions are also necessary. Recently Chowdhury et al.\cite{n} have studied stochastic comparison of the smallest ordered statistic when the parameters are vector majorized. In this paper we shall use the same notations as used in Chowdhury et al.\cite{n} for convenience. The Weibull-G distribution is an advanced distribution that can obtain the odds that an individual will die prior to time $X$, with cdf $F$, where these odds follow another life distribution $T$ as observed by Cooray \cite{4}. Bourguignon et al.\cite{11} have studied the mathematical properties of Weibull-G distribution and observed that this distribution nicely fits real data sets. \\

The Gompertz Makeham distribution evolved from the well known Gompertz law and Makeham law respectively, provided by two researchers \cite{12} and \cite{13} in the 19th century. It was observed that the death rate at any age is the sum of two terms in which one term contributes as age-independent (Makeham term) and the other term is the age dependent term (Gompertz term). \cite{14}, \cite{15}, \cite{16}, \cite{17}, \cite{18} have studied various interesting properties of Gompertz Makeham (GM) distribution. Let $X$ be a continuous random variable following GM distribution with positive parameters $\alpha, \beta, \lambda$, we shall denote this as $GM(\alpha, \beta, \lambda)$. The distribution function of $X$ is 
\begin{equation}
F_{GM}(x) = 1- e^{-\lambda x - \dfrac{\alpha}{\beta}(e^{\beta x}-1)}, x >0,  \alpha >0, \beta >0, \lambda >0.
\end{equation}
Here the parameter $\alpha$ denotes the initial mortality and $\beta$ denotes the increase in mortality with increasing age, both of them are Gompertz term whereas the parameter $\lambda$ is a Makeham term and it denotes the risk of death due to unexpected causes such as accidents, infections, weather conditions etc. In recent times, few extensions of Gompertz Makeham distribution are available such as the bivariate Gompertz Makeham life distribution \cite{20} and the transmuted Gompertz Makeham distribution \cite{21}, \cite{22}. In this paper we shall study the stochastic comparison results for series and parallel systems comprising of n components (say) where each component follow GM distribution.
 Many researchers have studied the stochastic comparison of Weibull distributions, Li and Li \cite{10}, Torrado and Kochar \cite{9} are among them. Fang and Zhang\cite{7} have compared two parallel systems comprising of Exponentiated Weibull components in terms of usual stochastic, dispersive and likelihood ratio ordering. Later Kundu and Chowdhury\cite{1} also studied parallel systems whose components follow Exponentiated Weibull distribution with respect to reversed hazard rate and likelihood ratio ordering. The above mentioned studies were carried out for vector- majorized parameters. The univariate majorization arises during income allocation or in comparison of only one attribute. In real life situations when allocation of more than one attribute are compared, one can use multivariate majorization. Studies have been conducted for multivariate chain majorization by various researchers. Fang and Balakrishnan\cite{3} studied Exponential-Weibull distribution and obtained usual stochastic and hazard rate ordering for smallest ordered statistic with vector majorized components and usual stochastic ordering of largest ordered statistic when the components were chain majorized. Biswas and Gupta\cite{8} observed usual stochastic ordering for parallel system with Exponentiated Gumbel Type-II distributed components for chain majorized parameters. \cite{23} observed various stochastic ordering for series and parallel systems whose components follow Kumaraswamy's and Frechet distribution. Moreover \cite{KS} obtained several stochastic ordering results for Kumaraswamy-G distributed components. Recently \cite{24} observed stochastic orderings for Extended Inverse Lindley distribution with heterogeneous components. The stochastic orders mentioned here are available in details in the book by Shaked and Shantikumar \cite{2b}. Some of the orders that have been used in the paper are mentioned in the next section. \\
In this paper we have discussed the stochastic comparison of the smallest and largest ordered statistic from independent heterogeneous Weibull-G random variables. Let $X_1, X_2, \ldots, X_n$ be independent random variables with $X_i \sim W-G(\alpha_i, \beta_i, \gamma_i), i=1,2,\ldots, n$. Furthermore, let $Y_1, Y_2, \ldots, Y_n$ be another set of independent random variables with $Y_i \sim W-G(\alpha_i^*, \beta_i^*, \gamma_i^*), i=1,2,\ldots, n$. When $\beta_1 = \beta_2 = \ldots = \beta_n= \beta_1^* = \beta_2^* = \ldots = \beta_n^*$ and the matrix containing the parameters $\alpha_i, \gamma_i$ changes to another matrix containing the parameters $\alpha_i^*, \gamma_i^*, i=1,2, \ldots, n $, in the sense of multivariate chain majorization, we study the hazard rate ordering of the smallest ordered statistic. Next, when $\beta_1 = \beta_2 = \ldots = \beta_n= \beta_1^* = \beta_2^* = \ldots = \beta_n^*$ and $\gamma_1 = \gamma_2 = \ldots = \gamma_n= \gamma_1^* = \gamma_2^* = \ldots = \gamma_n^*$ and $(\alpha_1, \alpha_2,\ldots, \alpha_n) \prec_{w} (\alpha_1^*, \alpha_2^*,\ldots, \alpha_n^*)$, we establish reversed hazard rate ordering of the largest ordered statistic. Also, when $\alpha_1 = \alpha_2 = \ldots = \alpha_n= \alpha_1^* = \alpha_2^* = \ldots = \alpha_n^*$, $\beta_1 = \beta_2 = \ldots = \beta_n= \beta_1^* = \beta_2^* = \ldots = \beta_n^*$ and $(\gamma_1, \gamma_2, \ldots, \gamma_n) \prec^w (\gamma_1^*, \gamma_2^*, \ldots, \gamma_n^*)$, we observe the usual stochastic ordering of the largest ordered statistic when the baseline distribution of $X$ is Exponential.\\
The paper is organized as follows: The definitions of various terms used in the paper are mentioned in Section 2. The results for smallest ordered statistic with chain majorized parameters and for largest ordered statistic with vector majorized parameters are discussed in Section 3 and 4. The conclusion is given in Section 5.

\section{Definitions}

We consider $X$ and $Y$ to be two absolutely continuous random variables with distribution functions $F(x)$ and $G(x)$; survival functions as $\overline{F}(x)$ and $\overline{G}(x)$; probability density functions as $f(x)$ and $g(x)$.
%; hazard rate functions as $r(x) = \dfrac{f(x)}{\overline{F}(x)}$ and $s(x) = \dfrac{g(x)}{\overline{G}(x)}$; reversed hazard rate functions as $\tilde{r}(x) = \dfrac{f(x)}{F(x)}$ and $\tilde{s}(x) = \dfrac{g(x)}{G(x)}$. 
Then the two random variables can be compared with each other with the help of various stochastic orders such as:
\begin{enumerate}
\item $X$ is smaller than $Y$ in the usual stochastic order ($X \leq_{st} Y$) if and only if $$\overline{F}(x) \leq \overline{G}(x) ~\forall x \in (-\infty, \infty).$$

\item $X$ is smaller than $Y$ in hazard rate order ($X \leq_{hr} Y$) if and only if 
 $$\dfrac{\overline{G}(x)}{\overline{F}(x)}  \text{ increases in }x \in (-\infty,max(u_X,u_Y))$$ where $u_X$ and $u_Y$ are the right end-points of the supports of $X$ and $Y$ respectively.

\item $X$ is smaller than $Y$ in the reversed hazard rate order $X \leq_{rh} Y$ if and only if 
$$\dfrac{G(x)}{F(x)} \text{ increases in}~ x \in (\min(l_X,l_Y),\infty)$$ where $l_X$ and $l_Y$ are the left end-points of the supports of $X$ and $Y$ respectively.
\item $X$ is smaller than $Y$ in the likelihood ratio order ($X \leq_{lr}Y$) if and only if $$\dfrac{g(x)}{f(x)} \text{ increases in x over the union of the supports of X and Y.}$$

\end{enumerate}

 The following concept of majorization is available in Marshall et al.(\cite{2}).
 \subsection{Vector Majorization}
Consider two n-dimensional real valued vectors $\underline{a} = (a_1, \ldots, a_n)$ and $\underline{b} = (b_1, \ldots, b_n)$, such that they are arranged as $a_{1:n} \leq a_{2:n} \leq \ldots \leq a_{n:n}$ and $b_{1:n} \leq b_{2:n} \leq \ldots \leq b_{n:n}$. Then the vector 
\begin{enumerate}
\item $\underline{a}$ is \emph{majorized} by $\underline{b}$   ( $\underline{a} \prec \underline{b}$ ) if
\begin{equation}
 \sum_{i=1}^{n}a_{i:n} = \sum_{i=1}^{n}b_{i:n} ~\text{and}~ \sum_{i=1}^{k}a_{i:n} \geq \sum_{i=1}^{k}b_{i:n} ~\forall~k = 1,\ldots, n-1;
 \end{equation}
\item $\underline{a}$ is \emph{weakly submajorized} by $\underline{b}$ ( $\underline{a} \prec_{w} \underline{b}$ ) if
\begin{equation}
  \sum_{i=1}^k a_{n-i+1:n} \leq \sum_{i=1}^k b_{n-i+1:n} ~\forall~ k = 1, \ldots, n;
\end{equation}
\item $\underline{a}$ is \emph{weakly supermajorized} by $\underline{b}$ ( $\underline{a} \prec^{w} \underline{b}$ ) if
\begin{equation}
\sum_{i=1}^k a_{i:n} \geq \sum_{i=1}^k b_{i:n} ~\forall~ k = 1, \ldots, n;
\end{equation}
\end{enumerate}
Clearly, $\underline{a} \prec^{w} \underline{b} \Leftarrow \underline{a} \prec \underline{b} \Rightarrow \underline{a} \prec_{w} \underline{b}$.

\subsection{Multivariate chain Majorization}
A permutation matrix $\Pi$ is a square matrix where each row and column has exactly one entry $``1"$ and all the other entries are zero. The identity matrix of order $n$ is a permutation matrix and thus interchanging its rows and columns one can obtain $n!$ permutation matrices. A T-transform matrix is of the form 
$$T^{\lambda}_{i,j} = \lambda I_n + (1-\lambda)\Pi_{i,j},$$
where $0 \leq \lambda \leq 1$, $\Pi_{i,j}$ is a $n \times n$ permutation matrix that interchanges the i-th row with the j-th row and $I_n$ is the nth order identity matrix. Let $T^{\lambda_1}_{i,j} = \lambda_1 I_n + (1-\lambda_1)\Pi_{i,j}
^{(1)}$ and $T^{\lambda_2}_{i,j} = \lambda_2 I_n + (1-\lambda_2)\Pi_{i,j}
^{(2)}$ be two T-transform matrices, where $0 \leq \lambda_1, \lambda_2\leq 1$. If 
$\Pi_{i,j}^{(1)} = \Pi_{i,j}^{(2)}$, then $T^{\lambda_1}_{i,j}$ and $T^{\lambda_2}_{i,j}$ have the same structure otherwise they are different. A square matrix of order n, $Q= \{q_{uv}\}$ is doubly stochastic if $\displaystyle\sum_{u=1}^n q_{uv} = 1$, $\forall v=1,2,\ldots,n$ and $\displaystyle\sum_{v=1}^n q_{uv} = 1$, $\forall u=1,2,\ldots,n$. 

Consider two $m \times n$ matrices $A=\{a_{ij}\}$ and $B= \{b_{ij}\}$, the respective rows are $a_1^{R}, \ldots, a_m^{R}$ and $b_1^{R}, \ldots, b_m^{R}$. Then 
\begin{enumerate}
\item $A$ is said to chain majorize $B$ ($ A \gg B$), if there exists a finite set of $n \times n$ T-transform matrices $T^{\lambda_1}, \ldots, T^{\lambda_k}$ such that $B = A T^{\lambda_1}, \ldots, T^{\lambda_k} $;
\item $A$ is said to majorize $B$ ($A \succ B$), if there exists a $n\times n$ doubly stochastic matrix $Q$ such that $B= AQ$.
\end{enumerate}
 Since product of T-transforms is a doubly stochastic matrix thus 
 $$ A \gg B \Rightarrow A \succ B.$$
 When $m \geq 2$ and $n \geq 3$, majorization does not imply chain majorization.

\subsection{Schur-convexity (Schur-concavity)}
A real valued function $\psi$ defined on a subset of $\mathbb{R}^n$ is \emph{Schur-convex (Schur-concave)} if
\begin{equation}
\underline{a} \prec \underline{b} \Rightarrow \psi(\underline{a}) ~\leq (\geq)~ \psi(\underline{b}),
\end{equation}
where $\underline{a} = (a_1, \ldots, a_n)$ and $\underline{b} = (b_1, \ldots, b_n)$ are two real valued vectors.

\begin{lem}[Theorem 3.A.4, see {Marshall et al.(1979)\cite{2}}]
Let $A \subset \mathbb{R}$ be an open interval , a function $\psi: A^n \rightarrow \mathbb{R}$ be continuously differentiable then $\psi$ is Schur-convex (Schur-concave) on $A^n$ if and only if $\psi$ is symmetric on $A^n$, and
\begin{equation*}
(a_i - a_j)\left(\dfrac{\partial \psi(\underline{a})}{\partial a_i} - \dfrac{\partial \psi(\underline{a})}{\partial a_j}\right) \geq (\leq) ~ 0 ~~~~\forall \underline{a} \in A^n,
\end{equation*}
where $\underline{a} = (a_1,\ldots, a_n)$.
\end{lem}

\begin{lem}[Proposition 3.C.1, see {Marshall et al.(1979)\cite{2}}]\label{lemma2}
	\normalfont If $\mathbb{A} \subset \mathbb{R}$ is an interval and $h: \mathbb{A} \rightarrow \mathbb{R}$ is convex (concave), then $\psi(\underline{a}) = \displaystyle\sum_{i=1}^n h(a_i)$ is Schur-convex (Schur-concave) on $\mathbb{A}^n$, where $\underline{a} = (a_1, \ldots, a_n)$.
\end{lem}
\begin{lem}[Theorem 3.A.8, see {Marshall et al.(1979)\cite{2}}] \label{lemma3}
\normalfont	Let $S \subset \mathbb{R}^n$, a function $f: S \rightarrow \mathbb{R}$ satisfying $$\underline{a} \prec_w \underline{b} ~(\underline{a} \prec^w \underline{b}) \mbox{ on S } \Rightarrow f(\underline{a}) \leq f(\underline{b})$$ if and only if $f$ is increasing (decreasing) and Schur-convex on $S$.
\end{lem}

\begin{lem}[Proposition 15.B.1, see {Marshall et al.(1979)\cite{2}}]
\normalfont A function $\psi: \mathbb{R}^4 \rightarrow \mathbb{R} $ is differentiable and satisfies
$$\psi(A) \leq (\geq) \psi(B) ~~~~ \forall A \prec \prec B $$
if and only if the following conditions are satisfied:
\begin{enumerate}
\item $\psi(A) = \psi(A \Pi)~~~\forall$ permutation matrices $\Pi$,
\item $\displaystyle\sum_{i=1}^2 (a_{ij} - a_{ik})\left(\dfrac{\partial \psi(A)}{\partial a_{ij}} - \dfrac{\partial \psi(A)}{\partial a_{ik}}\right) \geq (\leq) 0$ $~~~~~\forall j, k =1,2$.
\end{enumerate}
\end{lem}
\begin{lem}\label{lh1}
\normalfont Let $h_1:[0,\infty) \rightarrow \mathbb{R}$ be defined as
\begin{equation} \label{h1x}
h_1(x) = e^{x} - xe^{x}-1,
\end{equation}
then $h_1(x) \leq 0$ $\forall x >0$.
\end{lem}
\bf{Proof}
\normalfont The derivative of $h_1(x)$ with respect to $x$,
\begin{align*}
h_1^{\prime}(x) &= -xe^{x}\\
& <0.
\end{align*}
Therefore $h_1(x)$ is a decreasing function and $h_1(0)=0$, then $x>0 \Rightarrow h_1(x) \leq 0$.

\begin{lem}\label{lh2}
\normalfont Let $h_2:[0,\infty) \rightarrow \mathbb{R}$ be defined as
\begin{equation}\label{h2x}
h_2(x) = xe^{x}-2e^{x} +x+2,
\end{equation}
then $h_2(x) \geq 0$ $\forall x>0$.
\end{lem}
\bf{Proof}
\normalfont The derivative of $h_2(x)$ with respect to $x$,
\begin{align*}
h_2^{\prime}(x)&= xe^{x} -e^{x} +1\\
&=-h_1(x)\\
&>0.
\end{align*}
Therefore $h_2(x)$ is an increasing function and $h_2(0)=0$ then $x >0 \Rightarrow h_2(x) \geq 0$.\\

 Let us define
 $$P_n =
\begin{cases}
 (a,b) = \begin{bmatrix}
a_{1}&\cdots &a_{n} \\
b_{1}&\cdots &b_{n}
\end{bmatrix}: a_i, b_j >0, (a_i-a_j)(b_i-b_j)\geq 0, i,j = 1,2,\cdots,n.
\end{cases}
 $$
 This shall be used in the result section of the paper.
\section{Results}
Let $X_1, X_2, \ldots, X_n$ be independent random variables with $X_i \sim W-G(\alpha_i, \beta_i, \gamma_i), i=1,2,\ldots, n$ and let $Y_1, Y_2, \ldots, Y_n$ be another set of independent random variables with $Y_i \sim W-G(\alpha_i^*, \beta_i^*, \gamma_i^*), i=1,2,\ldots, n$. When $\beta_1 = \beta_2 = \ldots = \beta_n= \beta_1^* = \beta_2^* = \ldots = \beta_n^*$ and the matrix containing the parameters $\alpha_i, \gamma_i$ changes to another matrix containing the parameters $\alpha_i^*, \gamma_i^*, i=1,2, \ldots, n $,  in the sense of multivariate chain majorization, we study the hazard rate ordering of the smallest ordered statistic. The first result is studied when the value of $n=2$.
\subsection{Results for smallest ordered statistic with multivariate chain majorized components}
\begin{theorem}\label{thm1}
\normalfont Let $X_1,X_2$ and $Y_1,Y_2$ be 2 pairs of independent random variable such that $X_i \sim W-G(\alpha_i, \beta, \gamma_i)$ and $Y_i \sim W-G(\alpha_i^*, \beta, \gamma_i^*)$ for $i=1,2$.
Whenever $\beta \geq 2$, $w(x)$ is an increasing convex function, also $w^{\prime \prime}(x)$ is increasing and $\begin{bmatrix}
\alpha_1 & \alpha_2\\
\gamma_1 & \gamma_2
\end{bmatrix} \in P_2$, then 
\begin{equation*}
\begin{bmatrix}
\alpha_1 & \alpha_2\\
\gamma_1 & \gamma_2
\end{bmatrix} \prec \prec \begin{bmatrix}
\alpha_1^* & \alpha_2^*\\
\gamma_1^* & \gamma_2^*
\end{bmatrix} \Rightarrow X_{1:2} \geq_{hr} Y_{1:2}.
\end{equation*}

\end{theorem}
\begin{proof}
\normalfont
The reliability function of $X_{1:2}$ is
\begin{equation}
\overline{G}_{X_{1:2}}(x)=e^{\displaystyle-\sum_{i=1}^2 \alpha_i(w(\gamma_i x))^{\beta}}, ~\alpha_i >0, \beta >0, \gamma_i >0 ~~\forall i=1, 2, \ldots, n.
\end{equation}

Taking logarithm on both sides,
\begin{equation} \label{rWG}
-\ln\overline{G}_{X_{1:2}}(x) =  \displaystyle\sum_{i=1}^2 \alpha_i(w(\gamma_i x))^{\beta},
\end{equation}
differentiating \eqref{rWG} with respect to $x$, we obtain the hazard rate function as
$$r_{X_{1:2}}(x)= \beta \displaystyle\sum_{i=1}^2 \alpha_i \gamma_i (w(\gamma_i x))^{\beta -1}w'(\gamma_i x).$$
Consider $\psi(\underline{\alpha}, \underline{\gamma}) = (\alpha_1 - \alpha_2)\left(\dfrac{\partial r_{X_{1:2}}(x)}{\partial \alpha_1} - \dfrac{\partial r_{X_{1:2}}(x)}{\partial \alpha_2}\right)+(\gamma_1 - \gamma_2)\left(\dfrac{\partial r_{X_{1:2}}(x)}{\partial \gamma_1} - \dfrac{\partial r_{X_{1:2}}(x)}{\partial \gamma_2}\right)$.\\
Hence we have,
$$\psi(\underline{\alpha}, \underline{\gamma}) = \beta [(\alpha_1 - \alpha_2)(f(\gamma_1)-f(\gamma_2))+(\gamma_1 - \gamma_2)(\alpha_1 g(\gamma_1 x)-\alpha_2 g(\gamma_2 x))] ,$$
such that $f(\gamma) = \gamma(w(\gamma x))^{\beta - 1}w'(\gamma x)$ and $g(t) = (w(t))^{\beta -2}(w(t)w'(t)+ (\beta -1)t(w'(t))^2 + t w(t) w''(t))$. \\
We observe that
$$\dfrac{d}{d \gamma}f(\gamma) = (w(\gamma x))^{\beta -2} (w(\gamma x)w'(\gamma x)+ \gamma x (\beta -1)(w'(\gamma x))^2 + \gamma x w(\gamma x)w''(\gamma x))$$
 is positive for $\beta \geq 2$ and $w(.)$ is a convex and increasing function in its domain. Moreover,
 \begin{align*}
 g'(t) = &2(\beta -1)(w(t))^{\beta -2}(w'(t))^2 + t(\beta -1)(\beta -2)(w(t))^{\beta -3}(w'(t))^3 + 3(\beta -1)t (w(t))^{\beta -2}w'(t)w''(t) \\
 & + 2(w(t))^{\beta -1}w''(t)+ t (w(t))^{\beta -1} w'''(t)
 \end{align*}
 is positive as $\beta  \geq 2$ and $w'''(t) \geq 0$. \\
 Hence, $ (\underline{\alpha}, \underline{\gamma}) \in P_2 \Rightarrow \psi(\underline{\alpha}, \underline{\gamma}) \geq 0$. Using Lemma 2.4 we conclude here that $r_{X_{1:2}}(x) \leq r_{Y_{1:2}}(x) $, i.e., $X_{1:2} \geq_{hr} Y_{1:2}$.
 \end{proof}

 \begin{example}
\normalfont Let $X_1, X_2$ be independent random variables such that $X_i \sim W-Exp(\alpha_i, \beta, \gamma_i)$, $i=1,2$. Also let $Y_1, Y_2$ be another pair of independent random variable such that $Y_i \sim W-Exp(\alpha_i^*, \beta, \gamma_i^*), i=1,2$ (the baseline distribution is Exponential with cdf, $F(x) = 1- e^{-\gamma x})$. The parameters are given in the form of matrices as
 \begin{equation*}
 \begin{bmatrix}
\alpha_1^* & \alpha_2^*\\
\gamma_1^* & \gamma_2^*
\end{bmatrix} = \begin{bmatrix}
4.8 & 3.4\\
2.5 & 1.6
\end{bmatrix} \in P_2 \text{ and } \begin{bmatrix}
\alpha_1 & \alpha_2\\
\gamma_1 & \gamma_2
\end{bmatrix} = \begin{bmatrix}
4.03 & 4.17\\
2.005 & 2.095
\end{bmatrix}.
 \end{equation*}
 It can be observed that $\begin{bmatrix}
\alpha_1 & \alpha_2\\
\gamma_1 & \gamma_2
\end{bmatrix} = \begin{bmatrix}
\alpha_1^* & \alpha_2^*\\
\gamma_1^* & \gamma_2^*
\end{bmatrix} T^{0.45}$, and it satisfies the condition

 \begin{equation*}
\begin{bmatrix}
\alpha_1 & \alpha_2\\
\gamma_1 & \gamma_2
\end{bmatrix} \prec \prec \begin{bmatrix}
\alpha_1^* & \alpha_2^*\\
\gamma_1^* & \gamma_2^*
\end{bmatrix} \Rightarrow X_{1:2} \geq_{hr} Y_{1:2}.
\end{equation*}
The plot of the difference $r_{Y_{1:2}} - r_{X_{1:2}}$ is
\begin{figure}[h]
		\centering
			\includegraphics[scale=0.3]{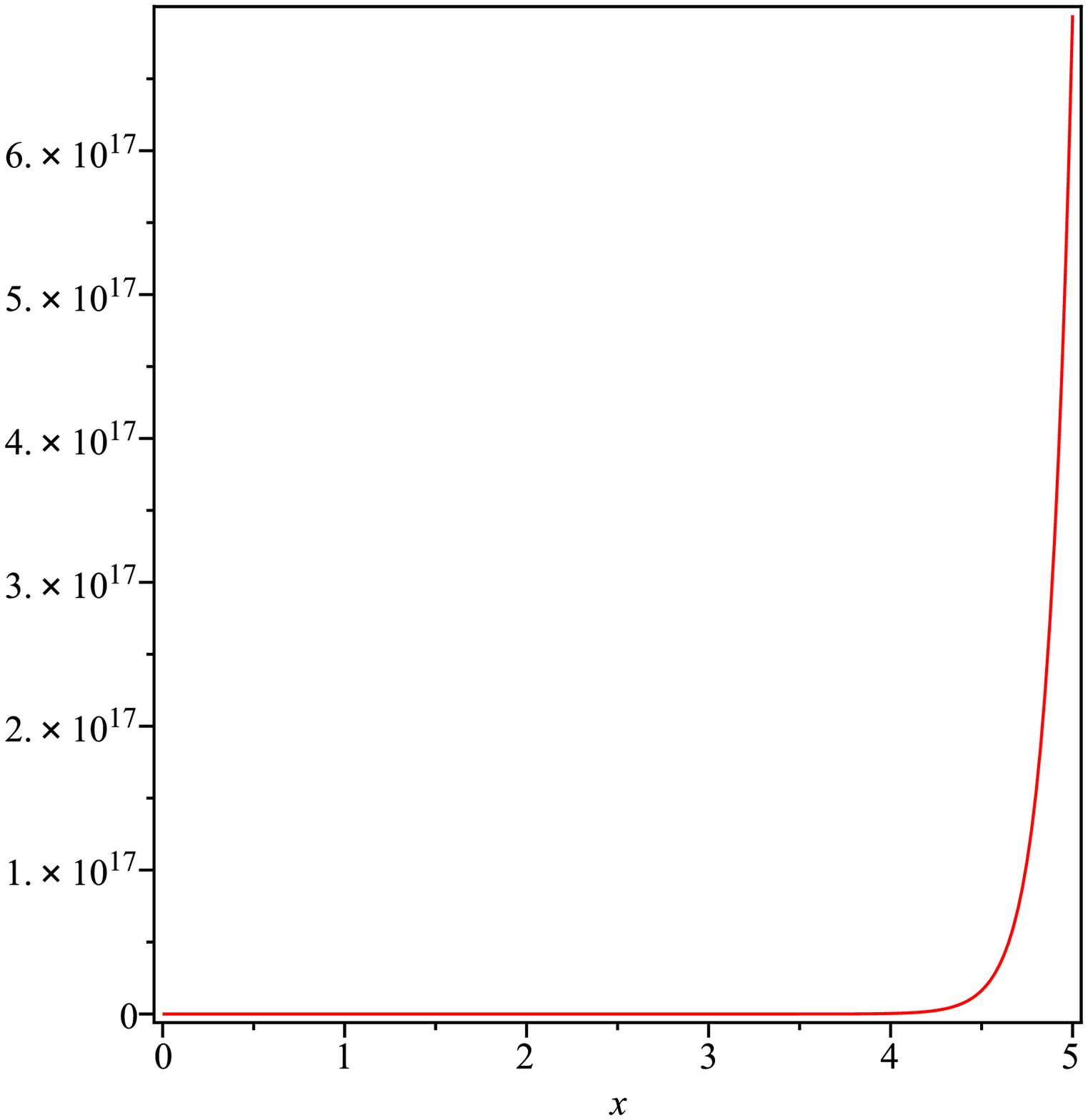}
			
			 \caption{The figure shows the graph of $r_{Y_{1:2}} - r_{X_{1:2}}$ for $\beta = 3$.}
	\end{figure}
	
 \end{example}

 The next result is observed for $n>2$.
 \begin{theorem}\label{thm2}
 \normalfont Let $X_1, \ldots, X_n$ be a set of independent random variables such that $X_i \sim W-G(\alpha_i, \beta, \gamma_i), i=1,\ldots,n$. Also let $Y_1,\ldots,Y_n$ be another set of random variable such that $Y_i \sim W-G(\alpha_i^*, \beta, \gamma_i^*), i=1,\ldots,n$. If
 \begin{equation*}
 \begin{bmatrix}
 \alpha_1 & \ldots & \alpha_n\\
 \gamma_1 & \ldots & \gamma_n
 \end{bmatrix} \in P_n
 \end{equation*}
 and
 \begin{equation*}
 \begin{bmatrix}
 \alpha_1^* & \ldots & \alpha_n^*\\
 \gamma_1^* & \ldots & \gamma_n^*
 \end{bmatrix} = \begin{bmatrix}
 \alpha_1 & \ldots & \alpha_n\\
 \gamma_1 & \ldots & \gamma_n
 \end{bmatrix} T_{i,j}^{\lambda},
 \end{equation*}
 then $\beta \geq 2, ~w^{\prime}(x) \geq 0, ~w^{\prime\prime}(x) \geq 0,  ~w^{\prime\prime\prime}(x) \geq 0$ $\Rightarrow X_{1:n} \geq_{hr} Y_{1:n}$.
 \end{theorem}
 \begin{proof}
 \normalfont We observe that $X_k$ and $Y_k$ have the same distribution (the parameters $\alpha_k = \alpha_k^*$ and $\gamma_k = \gamma_k^*$) $\forall k \neq i,j$ as the T-transform matrix $T_{i,j}^{\lambda} = \lambda I_{n} + (1-\lambda)\Pi_{i,j}$, where $\Pi_{i,j}$ interchanges the $i^{th}$ row with the $j^{th}$ row. Therefore applying Theorem \ref{thm1} the result follows.
 \end{proof}
 A finite product of T-transform matrices with the same structure is a T-transform matrix with the same structure. The above theorem can be applied in this case.  The finite product of T-transform matrix with different structures may or may not be a T-transform matrix in these circumstance the next theorem is useful.

 \begin{theorem}
 \normalfont Let $X_1, \ldots, X_n$ be a set of independent random variables such that $X_i \sim W-G(\alpha_i, \beta, \gamma_i), i=1,\ldots,n$. Also $Y_1,\ldots,Y_n$ be another set of random variable such that $Y_i \sim W-G(\alpha_i^*, \beta, \gamma_i^*), i=1,\ldots,n$. Presume $k \geq 2$, if
 \begin{equation*}
 \begin{bmatrix}
 \alpha_1 & \ldots & \alpha_n\\
 \gamma_1 & \ldots & \gamma_n
 \end{bmatrix} \in P_n,
 \end{equation*}
 \begin{equation*}
 \begin{bmatrix}
 \alpha_1 & \ldots & \alpha_n\\
 \gamma_1 & \ldots & \gamma_n
 \end{bmatrix} T^{\lambda_1}\ldots T^{\lambda_i} \in P_n, \text{for } i=1,\ldots, k-1,
 \end{equation*}
 and
 \begin{equation*}
 \begin{bmatrix}
 \alpha_1^* & \ldots & \alpha_n^*\\
 \gamma_1^* & \ldots & \gamma_n^*
 \end{bmatrix} = \begin{bmatrix}
 \alpha_1 & \ldots & \alpha_n\\
 \gamma_1 & \ldots & \gamma_n
 \end{bmatrix} T^{\lambda_1}\ldots T^{\lambda_k},
 \end{equation*}
 then $\beta \geq 2, ~w^{\prime}(x) \geq 0, ~w^{\prime\prime}(x) \geq 0,  ~w^{\prime\prime\prime}(x) \geq 0$ $\Rightarrow X_{1:n} \geq_{hr} Y_{1:n}$.
 \end{theorem}
 \begin{proof}
 \normalfont Let us fix
 \begin{equation}
 \begin{bmatrix}
 \alpha_1^{(j)} & \ldots & \alpha_n^{(j)}\\
 \gamma_1^{(j)} & \ldots & \gamma_n^{(j)}
 \end{bmatrix} = \begin{bmatrix}
 \alpha_1 & \ldots & \alpha_n\\
 \gamma_1 & \ldots & \gamma_n
 \end{bmatrix} T^{\lambda_1}\ldots T^{\lambda_j}, j=1, \ldots, k-1.
 \end{equation}
 Consider $Y_1^{(j)},\ldots, Y_n^{(j)}, j=1, \ldots, k-1$, be sets of independent random variables with $Y_i^{(j)} \sim W-G(\alpha_i^{(j)}, \beta, \gamma_i^{(j)}), i=1, \ldots, n$ and $j=1, \ldots, k-1, k \geq 2$. It has been assumed that
 \begin{equation}
 \begin{bmatrix}
 \alpha_1^{(j)} & \ldots & \alpha_n^{(j)}\\
 \gamma_1^{(j)} & \ldots & \gamma_n^{(j)}
 \end{bmatrix} \in P_n, \text{ for } j=1,\ldots, k-1.
 \end{equation}
 Using theorem $\ref{thm2}$ repeatedly we observe that $X_{1:n} \geq_{hr} Y_{1:n}^{(1)} \geq_{hr} \ldots Y_{1:n}^{(k-1)} \geq_{hr} Y_{1:n}$.
 \end{proof}
 \subsection{Results for largest ordered statistic with vector majorized components}
 When $\beta_1 = \beta_2 = \ldots = \beta_n= \beta_1^* = \beta_2^* = \ldots = \beta_n^*$ and $\gamma_1 = \gamma_2 = \ldots = \gamma_n= \gamma_1^* = \gamma_2^* = \ldots = \gamma_n^*$ and $(\alpha_1, \alpha_2,\ldots, \alpha_n) \prec_{w} (\alpha_1^*, \alpha_2^*,\ldots, \alpha_n^*)$, we observe the following result for  the largest ordered statistic.
 \begin{theorem}
\normalfont Let $X_1, X_2, \ldots, X_n$ be a set of independent random variables from Weibull- G distributed family of distribution such that $X_i \sim W-G(\alpha_i, \beta, \gamma)$ for $i=1,2, \ldots, n$. Let $Y_1, Y_2, \ldots, Y_n$ be another set of independent random variables such that $Y_i \sim W-G(\alpha_i^*, \beta, \gamma)$ for $i = 1, 2,\ldots , n$. Then $\underline{\alpha} \prec_{w} \underline{\alpha}^* \Rightarrow X_{n:n} \leq_{rh} Y_{n:n}$.
 \end{theorem}
 \begin{proof}
 \normalfont The distribution function of $X_{n:n}$ is
 \begin{equation}
 H_{X_{n:n}}(x)= \prod_{i=1}^n \left(1- e^{-\alpha_i(w(\gamma x))^\beta}\right), ~~x>0, ~\alpha_i >0, ~\beta>0, ~\gamma>0 ~\forall ~i=1,2,\ldots,n.
 \end{equation}
 Taking logarithm on both sides and differentiating with respect to $x$ we obtain the reversed hazard rate function as
 \begin{equation}\label{rhnn}
\tilde{r}_{X_{n:n}}(x) = \beta \gamma\left(\dfrac{F(\gamma x)}{1-F(\gamma x)}\right)^ {\beta -1} \dfrac{f(\gamma x)}{(1-F(\gamma x))^2} \sum_{i=1}^n \dfrac{\alpha_i e^{-\alpha_i (w(\gamma x))^{\beta}}}{1- e^{-\alpha_i (w(\gamma x))^{\beta}}}.
 \end{equation}
 For an easier understanding we rewrite equation \eqref{rhnn} as
 \begin{equation}
 \tilde{r}_{X_{n:n}}(x) =  \beta \gamma\left(\dfrac{F(\gamma x)}{1-F(\gamma x)}\right)^ {\beta -1} \dfrac{f(\gamma x)}{(1-F(\gamma x))^2} \sum_{i=1}^n g(\alpha_i),
 \end{equation}
 such that $g(\alpha) = \dfrac{\alpha}{e^{\alpha z} -1}$, where $z(x, \beta, \gamma) \equiv z = (w(\gamma x))^{\beta}$.
 It is required to show that the function $g(\alpha)$ is decreasing and convex.
 Computing $g^{\prime}(\alpha)$ and $g^{\prime\prime}(\alpha)$ we observe that,
 \begin{align}
 g^{\prime}(\alpha) &= \dfrac{e^{\alpha z}- \alpha ze^{\alpha z} - 1}{(e^{\alpha z}-1)^2} \label{g1},\\
 g^{\prime\prime}(\alpha) &= \dfrac{ze^{\alpha z}(\alpha ze^{\alpha z}-2e^{\alpha z}+ \alpha z +2)}{(e^{\alpha z}-1)^3} \label{g2}.
 \end{align}
 From \eqref{g1}, $g^{\prime}(\alpha) \overset{sign}{=} h_1(\alpha z) = e^{\alpha z}- \alpha ze^{\alpha z} - 1$ and $h_1(t)$ is a decreasing function of $t$. We recall lemma 2.5, then $\alpha > 0 \Rightarrow g^{\prime}(\alpha) \leq 0$.
 Similarly from \eqref{g2}, $g^{\prime\prime}(\alpha) \overset{sign}{=} h_2(\alpha z) = \alpha ze^{\alpha z}-2e^{\alpha z}+ \alpha z +2$. Using lemma $\ref{lh2}$ we observe that $g^{\prime\prime}(\alpha) \geq 0$. Thus $g(\alpha)$ is a decreasing convex function. Hence using lemma 2.2, we conclude that $\tilde{r}_{X_{n:n}}(x)$ is a Schur-convex function, therefore, we obtain $\underline{\alpha} \prec_{w} \underline{\alpha}^* \Rightarrow \tilde{r}_{X_{n:n}}(x) \leq_{rh} \tilde{r}_{Y_{n:n}}(x)$.
 \end{proof}

 When $\alpha_1 = \alpha_2 = \ldots = \alpha_n= \alpha_1^* = \alpha_2^* = \ldots = \alpha_n^*$, $\beta_1 = \beta_2 = \ldots = \beta_n= \beta_1^* = \beta_2^* = \ldots = \beta_n^*$ and $(\gamma_1, \gamma_2, \ldots, \gamma_n) \prec^w (\gamma_1^*, \gamma_2^*, \ldots, \gamma_n^*)$, we observe the usual stochastic ordering of the largest ordered statistic where the baseline distribution of $X$ is Exponential.

 \begin{theorem}
 \normalfont Let $X_1,X_2, \ldots,X_n$ be a set of independent random variables such that
$X_i \sim W-Exp(\alpha,\beta,\gamma_i)$ for $i = 1, 2, \ldots, n$. Let $Y_1, Y_2, \ldots , Y_n$ be another set of independent random
variables such that $Y_i \sim W-Exp(\alpha,\beta, \gamma_i^* )$ for $i = 1, 2, \ldots, n$. Then $\underline{\gamma} \prec^w \underline{\gamma}^* \Rightarrow X_{n:n} \leq_{st} Y_{n:n}$.
 \end{theorem}
 \begin{proof}
 \normalfont The baseline distribution is Exponential of the form $F(\gamma x) = 1- e^{-\gamma x}, x>0$. Therefore the function $w(\gamma x) = e^{\gamma x} -1$.
 The distribution function of $X_{n:n}$ is given by
 \begin{equation}\label{baseexp}
 F_{X_{n:n}}(x) = \prod_{i=1}^n \left(1-e^{-\alpha(e^{\gamma_i x}-1)^{\beta}}\right), x>0, ~\alpha>0, ~\beta>0, ~\gamma_i >0 ~\forall ~i=1,2,\ldots, n.
 \end{equation}
 Differentiating $\eqref{baseexp}$ with respect to $\gamma_i$,
 \begin{align*}
 \dfrac{\partial F_{X_{n:n}}(x)}{\partial \gamma_i} &= \alpha \beta x F_{X_{n:n}}(x) \dfrac{(e^{\gamma_i x}-1)^{\beta -1}e^{\gamma_i x}}{e^{\alpha(e^{\gamma_i x}-1)^{\beta}}-1}\\
 & > 0.
 \end{align*}

 Let $\psi_1(\gamma) = \dfrac{(e^{\gamma x}-1)^{\beta -1}e^{\gamma x}}{e^{\alpha(e^{\gamma x}-1)^{\beta}}-1}$, differentiating $\psi_1(\gamma)$ with respect to $\gamma$, we obtain
 \begin{align*}
 \psi_1^{\prime}(\gamma) &= -\dfrac{x(e^{\gamma x}-1)^{\beta-2}e^{\gamma x}}{(e^{\alpha(e^{\gamma x}-1)^{\beta}}-1)^2}(((\alpha \beta (e^{\gamma x}-1)^{\beta}-\beta)e^{\gamma x}+1)e^{\alpha(e^{\gamma x}-1)^{\beta}}+\beta e^{\gamma x} -1)\\
 & \overset{sign}{=} - (((\alpha \beta (e^{\gamma x}-1)^{\beta}-\beta)e^{\gamma x}+1)e^{\alpha(e^{\gamma x}-1)^{\beta}}+\beta e^{\gamma x} -1), \gamma >0, x>0.
 \end{align*}

Let us assume $e^{\gamma x} =t$, then the above equation (ignoring the -ve sign )can be rewritten as,
$$\psi_2(t)= ((\alpha \beta (t-1)^{\beta}-\beta)t+1)e^{\alpha(t-1)^{\beta}}+\beta t -1, ~t>1,$$
also $\psi_2(1) = 0$. Differentiating $\psi_2(t)$ with respect to $t$,
$$\psi_2^{\prime}(t)= \dfrac{\beta}{t-1}((({\alpha}^2\beta(t-1)^{2\beta}+ \alpha(t-1)^{\beta}-1)t+1)e^{\alpha(t-1)^{\beta}}+t-1).$$
Again consider,
$$\psi_3(t) = (({\alpha}^2\beta(t-1)^{2\beta}+ \alpha(t-1)^{\beta}-1)t+1)e^{\alpha(t-1)^{\beta}}+t-1,$$
differentiating $\psi_3(t)$ we obtain
\begin{align*}
\psi_3^{\prime}(t)&= 1-e^{\alpha(t-1)^{\beta}}+ e^{\alpha(t-1)^{\beta}}(\alpha(t-1)^{\beta}+{\alpha}^2\beta(t-1)^{2\beta}+2{\alpha}^2{\beta}^2t(t-1)^{2\beta-1}+{\alpha}^2{\beta}^2t(t-1)^{3\beta -1}\\
& ~~+{\alpha}^2\beta t(t-1)^{2\beta -1}+\alpha\beta{(t-1)}^{\beta -1}).
\end{align*}
Let us consider $\phi(z) = 1-e^{z}+ z e^{z}$, where $z= \alpha(t-1)^{\beta}$. Using lemma 2.5, we observe that $\phi(z)>0 ~~\forall z >0$, i.e., $\psi_3(t) >0$. Consequently $\psi_2^{\prime}(t) >0$ and for $t>1$, $\psi_2(t) >0$. Hence $\psi_1^{\prime}(\gamma) <0$ for $\gamma >0$.
 Consider
 \begin{align*}
 \Delta &= (\gamma_i -\gamma_j)\left(\dfrac{\partial F_{X_{n:n}}(x)}{\partial \gamma_i} - \dfrac{\partial F_{X_{n:n}}(x)}{\partial \gamma_j}\right)\\
 &= \alpha \beta x F_{X_{n:n}}(x) (\gamma_i -\gamma_j)(\psi_1(\gamma_i)-\psi_1(\gamma_j))\\
 & \leq 0,
 \end{align*}
 i.e., $F_{X_{n:n}}(x)$ is increasing and Schur-concave function with respect to the parameter $\gamma_i$ $\forall i =1,2, \ldots, n$. Hence $\underline{\gamma} \prec^w \underline{\gamma}^* \Rightarrow F_{X_{n:n}}(x) \geq F_{Y_{n:n}}(x)$.
 \end{proof}

\section{Gompertz-Makeham distribution}
In this section we shall discuss the stochastic ordering results for Gompertz Makeham distribution. When the parameter $\lambda$ is kept constant and the remaining two parameters are varied in the sense of multivariate chain majorization we observe  a hazard rate ordering between $X_{1:2}$ and $Y_{1:2}$.
\subsection{Results for smallest ordered statistic with multivariate chain majorized components}
\begin{theorem}
\normalfont Let $X_1,X_2$ and $Y_1,Y_2$ be 2 pairs of independent random variable such that $X_i \sim GM(\alpha_i, \beta_i, \lambda)$ and $Y_i \sim GM(\alpha_i^*, \beta_i^*, \lambda)$ for $i=1,2$. Let $\begin{bmatrix}
\alpha_1 & \alpha_2\\
\beta_1 & \beta_2
\end{bmatrix} \in P_2$, then
$$\begin{bmatrix}
\alpha_1 & \alpha_2\\
\beta_1 & \beta_2
\end{bmatrix} \prec \prec \begin{bmatrix}
\alpha_1^* & \alpha_2^*\\
\beta_1^* & \beta_2^*
\end{bmatrix} \Rightarrow X_{1:2} \geq_{hr} Y _{1:2}.$$
\end{theorem}
\begin{proof}
\normalfont The reliability function of the smallest ordered statistic $X_{1:2}$ is
\begin{equation}
\overline{F}_{X_{1:2}}(x) = e^{-2\lambda x}\prod_{i=1}^2 e^{-\dfrac{\alpha_i}{\beta_i}(e^{\beta_i x}-1)}, x> 0, \alpha_i>0, \beta_i >0, \lambda>0. 
\end{equation}
The corresponding hazard rate function is given by
\begin{equation}
r_{X_{1:2}}(x) = 2\lambda + \sum_{i=1}^2 \alpha_ie^{\beta_i x} , x>0.
\end{equation}
Consider 
\begin{align*}
\phi(\underline{\alpha},\underline{\beta}) &= (\alpha_1 - \alpha_2)\left(\dfrac{\partial r_{X_{1:2}}(x) }{\partial \alpha_1} - \dfrac{\partial r_{X_{1:2}}(x) }{\partial \alpha_2}\right) + (\beta_1 - \beta_2)\left(\dfrac{\partial r_{X_{1:2}}(x) }{\partial \beta_1}- \dfrac{\partial r_{X_{1:2}}(x) }{\partial \beta_2}\right)\\
&= (\alpha_1 - \alpha_2)(e^{\beta_1 x} - e^{\beta_2 x})+ x(\beta_1 -\beta_2)(\alpha_1e^{\beta_1 x} - \alpha_2e^{\beta_2 x})\\
& \geq 0.
\end{align*}
Since we have considered the parameter matrix $\begin{bmatrix}
\alpha_1 & \alpha_2\\
\beta_1 & \beta_2
\end{bmatrix} \in P_2$, then either $\alpha_1 > \alpha_2, \beta_1 > \beta_2$ or $\alpha_1 < \alpha_2, \beta_1 < \beta_2$. In both the cases the sign of $\phi(\underline{\alpha},\underline{\beta})$ is always positive. Hence we can conclude using Lemma 2.4 that $\begin{bmatrix}
\alpha_1 & \alpha_2\\
\beta_1 & \beta_2
\end{bmatrix} \prec \prec \begin{bmatrix}
\alpha_1^* & \alpha_2^*\\
\beta_1^* & \beta_2^*
\end{bmatrix} \Rightarrow r_{X_{1:2}} \leq r_{Y _{1:2}}$, i.e., $X_{1:2} \geq_{hr} Y _{1:2}$.

\end{proof}
As in Example 1, with the same set of parameters, we can observe the following example.
 \begin{example}
\normalfont Let $X_1, X_2$ be independent random variables such that $X_i \sim GM(\alpha_i, \beta_i, \lambda)$, $i=1,2$. Also let $Y_1, Y_2$ be another pair of independent random variable such that $Y_i \sim GM(\alpha_i^*, \beta_i^*, \lambda), i=1,2$. The parameters satisfy all the given conditions and are given in the form of matrices as
 \begin{equation*}
 \begin{bmatrix}
\alpha_1^* & \alpha_2^*\\
\beta_1^* & \beta_2^*
\end{bmatrix} = \begin{bmatrix}
4.8 & 3.4\\
2.5 & 1.6
\end{bmatrix} \in P_2 \text{ and } \begin{bmatrix}
\alpha_1 & \alpha_2\\
\beta_1 & \beta_2
\end{bmatrix} = \begin{bmatrix}
4.03 & 4.17\\
2.005 & 2.095
\end{bmatrix}.
 \end{equation*}
 The plot of the difference $r_{Y_{1:2}} - r_{X_{1:2}}$ is
\begin{figure}[h]
		\centering
			\includegraphics[scale=0.3]{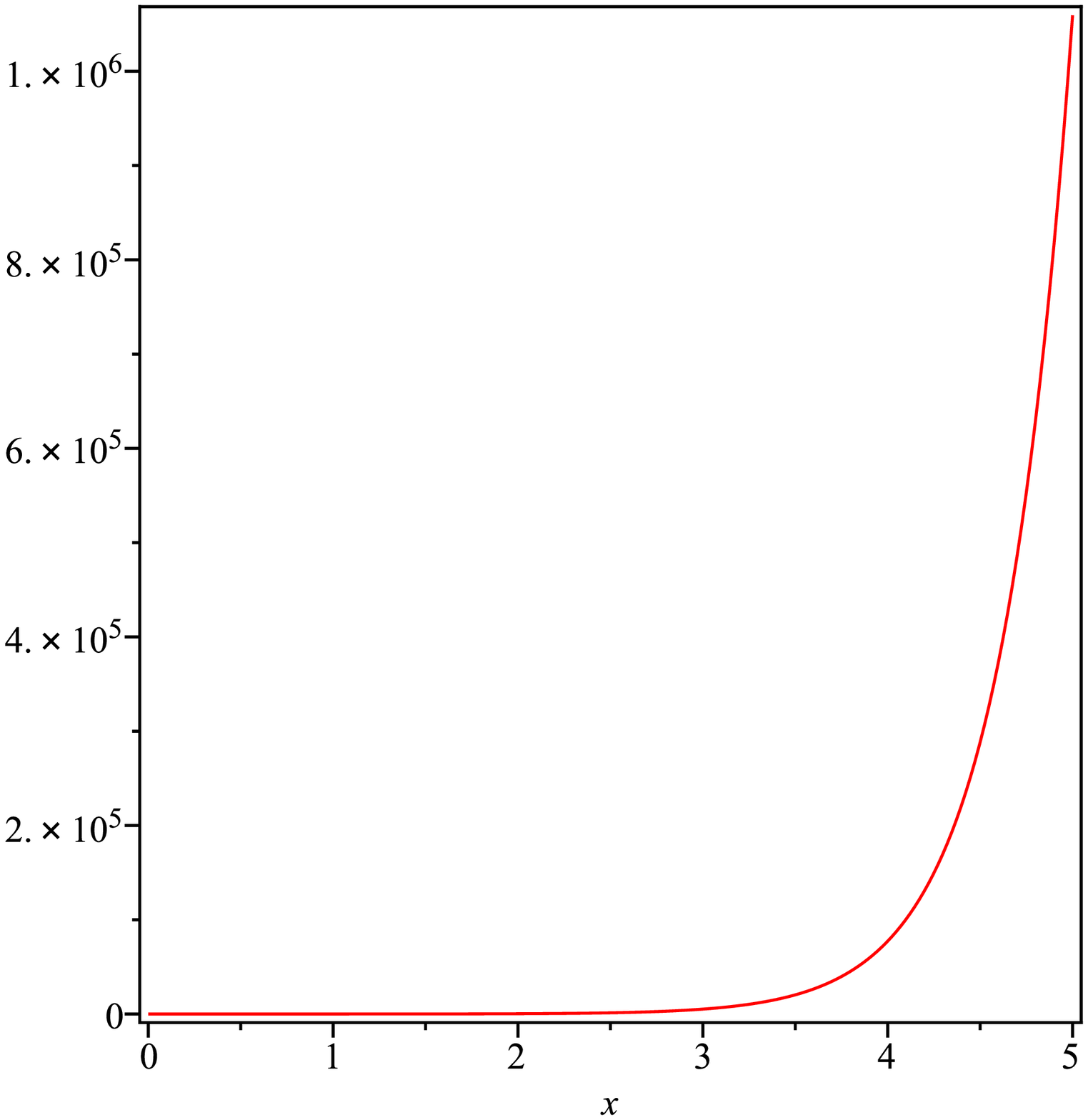}
			
			 \caption{The figure shows the graph of $r_{Y_{1:2}} - r_{X_{1:2}}$ for any value of $\lambda$.}
	\end{figure}
	
 \end{example}
The above theorem can be extended to n random variables using the next theorem.
\begin{theorem}
 \normalfont Let $X_1, \ldots, X_n$ be as set of independent random variables such that $X_i \sim GM(\alpha_i, \beta_i, \lambda), i=1,\ldots,n$. Also $Y_1,\ldots,Y_n$ be another set of random variable such that $Y_i \sim GM(\alpha_i^*, \beta_i^*, \lambda), i=1,\ldots,n$. If
 \begin{equation*}
 \begin{bmatrix}
 \alpha_1 & \ldots & \alpha_n\\
 \lambda_1 & \ldots & \lambda_n
 \end{bmatrix} \in P_n
 \end{equation*}
 and
 \begin{equation*}
 \begin{bmatrix}
 \alpha_1^* & \ldots & \alpha_n^*\\
 \lambda_1^* & \ldots & \lambda_n^*
 \end{bmatrix} = \begin{bmatrix}
 \alpha_1 & \ldots & \alpha_n\\
 \lambda_1 & \ldots & \lambda_n
 \end{bmatrix} T_{i,j}^{\delta},
 \end{equation*}
 then $X_{1:n} \geq_{hr} Y_{1:n}$.
 \end{theorem}
 \begin{proof}
 \normalfont Proceeding in a similar way as in Theorem 3.2, the parameters $\alpha_k = \alpha_k^*$ and $\lambda_k = \lambda_k^*$ corresponding to the r.v. $X_k$ and $Y_k$ respectively $\forall k \neq i,j$. The T-transform matrix $T_{i,j}^{\delta} = \delta I_{n} + (1-\delta)\Pi_{i,j}$ where $\Pi_{i,j}$ interchanges the $i^{th}$ row with the $j^{th}$ row. Therefore the result follows from Theorem 4.1.
 \end{proof}
 
 \begin{theorem}
 \normalfont Let $X_1,\ldots, X_n$ be a set of independent random variables such that $X_i \sim GM(\alpha_i, \beta_i, \lambda), i=1,\ldots,n$. Also let $Y_1, \ldots, Y_n$ be another set of random variables such that $Y_i \sim GM(\alpha_i^*, \beta_i^*,\lambda), i=1,\ldots,n$. For $k\geq 2$, if 
 \begin{equation*}
 \begin{bmatrix}
 \alpha_1 & \ldots & \alpha_n\\
 \beta_1 & \ldots & \beta_n
 \end{bmatrix} \in P_n,
 \end{equation*}
 \begin{equation*}
 \begin{bmatrix}
 \alpha_1 & \ldots & \alpha_n\\
 \beta_1 & \ldots & \beta_n
 \end{bmatrix} T^{\delta_1}\ldots T^{\delta_i} \in P_n, \text{for } i=1,\ldots, k-1,
 \end{equation*}
 and
 \begin{equation*}
 \begin{bmatrix}
 \alpha_1^* & \ldots & \alpha_n^*\\
 \beta_1^* & \ldots & \beta_n^*
 \end{bmatrix} = \begin{bmatrix}
 \alpha_1 & \ldots & \alpha_n\\
 \beta_1 & \ldots & \beta_n
 \end{bmatrix} T^{\delta_1}\ldots T^{\delta_k},
 \end{equation*}
 then $X_{1:n} \geq_{hr} Y_{1:n}$.
 \end{theorem}
 \begin{proof}
 \normalfont Consider 
 \begin{equation}
 \begin{bmatrix}
 \alpha_1^{(j)} & \ldots & \alpha_n^{(j)}\\
 \beta_1^{(j)} & \ldots & \beta_n^{(j)}
 \end{bmatrix} = \begin{bmatrix}
 \alpha_1 & \ldots & \alpha_n\\
 \beta_1 & \ldots & \beta_n
 \end{bmatrix} T^{\delta_1}\ldots T^{\delta_j}, j=1,\ldots, k-1.
 \end{equation}
 Let us assume $Y_1^{(j)},\ldots, Y_n^{(j)}, j=1,\ldots, k-1$, be sets of independent random variables with $Y_i^{(j)} \sim GM(\alpha_i^{(j)}, \beta_i^{(j)}, \lambda)$, $i=1,\ldots,n$ and $j=1,\ldots, k-1, k \geq 2$. Since 
$$\begin{bmatrix}
\alpha_1^{(j)} & \ldots & \alpha_n^{(j)}\\
 \beta_1^{(j)} & \ldots & \beta_n^{(j)}
 \end{bmatrix} \in P_n, \text{ for } j=1,\ldots, k-1.$$
 Thus using Theorem 4.2 repeatedly we obtain $X_{1:n} \geq_{hr} Y_{1:n}^{(1)} \geq_{hr} \ldots, Y_{1:n}^{(k-1)} \geq_{hr} Y_{1:n}$.
 \end{proof}
\subsection{When the components are chain majorized}
In this section we shall observe the stochastic behaviour of GM distributed components when only one parameter is varied and all the other parameters are kept constant. Firstly we shall observe the behaviour of the parameter $\lambda$ in the sense of usual stochastic ordering.
\begin{theorem}
\normalfont Let $X_1,X_2, \ldots,X_n$ be a set of independent random variables such that
$X_i \sim GM(\alpha,\beta,\lambda_i)$ for $i = 1, 2, \ldots, n$. Let $Y_1, Y_2, \ldots , Y_n$ be another set of independent random
variables such that $Y_i \sim GM(\alpha,\beta, \lambda_i^* )$ for $i = 1, 2, \ldots, n$. Then $\underline{\lambda} \prec^{w} \underline{\lambda}^* \Rightarrow X_{1:n} =_{st} Y_{1:n}$.
\end{theorem}
\begin{proof}
\normalfont The survival function of $X_{1:n}$ is 
\begin{equation}
\overline{F}_{X_{1:n}}(x) = e^{-\left(\displaystyle\sum_{i=1}^n \lambda_i\right) x} e^{-n \dfrac{\alpha}{\beta}(e^{\beta x}-1)}, x> 0, \alpha>0, \beta >0, \lambda_i>0.
\end{equation}
Consider $\phi_1(\underline{\lambda}) = (\lambda_i -\lambda_j)\left(\dfrac{\partial \overline{F}_{X_{1:n}}(x)}{\partial \lambda_i} - \dfrac{\partial \overline{F}_{X_{1:n}}(x)}{\partial \lambda_j}\right)$, where $i \neq j$. Now $\dfrac{\partial \overline{F}_{X_{1:2}}(x)}{\partial \lambda_i} = -x \overline{F}_{X_{1:2}}(x)$. We can thus observe that $\phi_1(\underline{\lambda}) =0$ for every $\alpha >0, \beta>0$ and $\overline{F}_{X_{1:2}}(x)$ is decreasing with respect to each $\lambda_i$, $i=1,2,\ldots,n$. Using lemma 2.1 the result follows.
\end{proof}

In a similar manner, we have obtained the stochastic comparison results for the maximum ordered statistic, $X_{n:n}$ also. The following theorem  presents a usual stochastic ordering between $X_{n:n}$ and $Y_{n:n}$.
\begin{theorem}
\normalfont Let $X_1,X_2, \ldots,X_n$ be a set of independent random variables such that
$X_i \sim GM(\alpha_i,\beta,\lambda)$ for $i = 1, 2, \ldots, n$. Let $Y_1, Y_2, \ldots , Y_n$ be another set of independent random
variables such that $Y_i \sim GM(\alpha_i^*,\beta, \lambda)$ for $i = 1, 2, \ldots, n$. Then $\underline{\alpha} \prec^{w} \underline{\alpha}^* \Rightarrow X_{n:n} \leq_{st} Y_{n:n}$.
\end{theorem}
\begin{proof}
\normalfont The distribution function of $X_{n:n}$ is
\begin{equation}
F_{X_{n:n}}(x)= \prod_{i=1}^n \left( 1- e^{-\lambda x - \dfrac{\alpha_i}{\beta}(e^{\beta x}-1)} \right), x >0, \alpha_i >0, \beta>0, \lambda >0.
\end{equation}
Consider $\phi_2(\underline{\alpha}) = (\alpha_i -\alpha_j)\left(\dfrac{\partial F_{X_{n:n}}(x)}{\partial \alpha_i} - \dfrac{\partial F_{X_{n:n}}(x)}{\partial \alpha_j}\right)$, for $i \neq j$. Now, 
$$\dfrac{\partial F_{X_{n:n}}(x)}{\partial \alpha_i} = \dfrac{1}{\beta} \dfrac{F_{X_{n:n}}(x) (e^{\beta x}-1)}{e^{\lambda x + \dfrac{\alpha_i}{\beta}(e^{\beta x}-1)}-1}.$$
Thus, $$\phi_2(\underline{\alpha}) = \dfrac{1}{\beta}(e^{\beta x}-1)F_{X_{n:n}}(x)(\alpha_i -\alpha_j)\left(p(\alpha_i) - p(\alpha_j)\right), \alpha_i \neq \alpha_j$$ where $p(\alpha) = \dfrac{1}{e^{\lambda x+\dfrac{\alpha}{\beta}(e^{\beta x}-1)}-1}$.\\
 We observe that the quantity $p(\alpha)$ is decreasing with respect to $\alpha$, since\\
  $p^{\prime}(\alpha) = -\left(e^{\lambda x+\dfrac{\alpha}{\beta}(e^{\beta x}-1)}-1\right)^{-2}\dfrac{(e^{\beta x}-1)}{\beta}e^{\lambda x + \dfrac{\alpha}{\beta}(e^{\beta x}-1)}$. Consequently, $\phi_2(\underline{\alpha}) \leq 0$. Hence using lemma 2.1, $F_{X_{n:n}}(x)$ is Schur-concave with respect to $\underline{\alpha} = (\alpha_1,\alpha_2,\ldots,\alpha_n)$. Also, $F_{X_{n:n}}(x)$ is increasing w.r.t each $\alpha_i$, $i=1,2,\ldots,n$. In other words $-F_{X_{n:n}}(x)$ is Schur-convex and decreasing with respect to $\underline{\alpha}$. Using Lemma 2.3, we observe that $\underline{\alpha} \prec^{w} \underline{\alpha}^* \Rightarrow F_{X_{n:n}}(x) \geq_{st} F_{Y_{n:n}}(x)$ and the result follows.
\end{proof}
 \section{Conclusion}
 We have observed the following results:
  Let $X_1, X_2, \ldots, X_n$ be independent random variables with $X_i \sim W-G(\alpha_i, \beta_i, \gamma_i), i=1,2,\ldots, n$. Furthermore, let $Y_1, Y_2, \ldots, Y_n$ be another set of independent random variables with $Y_i \sim W-G(\alpha_i^*, \beta_i^*, \gamma_i^*), i=1,2,\ldots, n$. When $\beta_1 = \beta_2 = \ldots = \beta_n= \beta_1^* = \beta_2^* = \ldots = \beta_n^*$ and the matrix containing the parameters $\alpha_i, \gamma_i$ changes to another matrix containing the parameters $\alpha_i^*, \gamma_i^*, i=1,2, \ldots, n $ in the sense of multivariate chain majorization, we study the hazard rate ordering of the smallest ordered statistic.\\
  Next, when $\beta_1 = \beta_2 = \ldots = \beta_n= \beta_1^* = \beta_2^* = \ldots = \beta_n^*$ and $\gamma_1 = \gamma_2 = \ldots = \gamma_n= \gamma_1^* = \gamma_2^* = \ldots = \gamma_n^*$ and $(\alpha_1, \alpha_2,\ldots, \alpha_n) \prec_{w} (\alpha_1^*, \alpha_2^*,\ldots, \alpha_n^*)$, we establish reversed hazard rate ordering of the largest ordered statistic. Also when $\alpha_1 = \alpha_2 = \ldots = \alpha_n= \alpha_1^* = \alpha_2^* = \ldots = \alpha_n^*$, $\beta_1 = \beta_2 = \ldots = \beta_n= \beta_1^* = \beta_2^* = \ldots = \beta_n^*$ and $(\gamma_1, \gamma_2, \ldots, \gamma_n) \prec^w (\gamma_1^*, \gamma_2^*, \ldots, \gamma_n^*)$, we observe the usual stochastic ordering of the largest ordered statistic when the baseline distribution of $X$ is Exponential.\\
  A similar set of random variables following Gompertz Makeham distribution $(X_i \sim GM(\alpha_i, \beta_i, \lambda_i) \text{ and } Y_i \sim GM(\alpha_i^*, \beta_i^*, \lambda_i^*))$ are considered. We observed hazard rate ordering for the sample minimum when the parameters $\lambda_1 =\lambda_2=\ldots, \lambda_n = \lambda_1^* =\lambda_2^*=\ldots, \lambda_n^* = \lambda$ and the other remaining parameters are related in the sense of multivariate chain majorization. Whereas usual stochastic ordering has been observed for the sample maximums when the parameters $\beta_1=\beta_2=\ldots, \beta_n = \beta_1^*=\beta_2^*=\ldots, \beta_n^*$ and $\lambda_1=\lambda_2=\ldots, \lambda_n = \lambda_1^*=\lambda_2^*=\ldots, \lambda_n^*$ while the parameter $\alpha$ is only varied using vector majorization. Also we observe that the age independent parameter $\lambda$ has absolutely no effect on the ordered statistics. 
  
 \section*{Acknowledgments}
  The first author would like to thank IIT Kharagpur for research assistantship.

\end{document}